\documentclass[10pt, final]{amsart} 

\usepackage{amsmath, amssymb, amscd, latexsym}

\usepackage[notref, notcite]{showkeys}

\renewcommand{\geq}{\geqslant}
\renewcommand{\leq}{\leqslant}

\renewcommand{\H}{\mathcal{H}}

\newcommand{\R}{\mathbb{R}}
\newcommand{\C}{\mathbb{C}}
\renewcommand{\P}{\mathbb{P}}

\newtheorem{theorem}{Theorem}[section]
\newtheorem{lemma}[theorem]{Lemma}
\newtheorem{proposition}[theorem]{Proposition}

\newtheorem{corollary}[theorem]{Corollary}
\newtheorem*{main-theorem}{Main Theorem}
\newtheorem*{remark*}{Remark}
\numberwithin{equation}{section}

\title[Solitary water waves]{No solitary waves exist on 2D deep water}

\author[Hur]{Vera~Mikyoung~Hur}
	\address{Department of Mathematics, University of Illinois at Urbana-Champaign, Urbana, IL 61801}
	\email{verahur@\allowbreak math.\allowbreak uiuc.\allowbreak edu}

\begin{document}

\maketitle

\begin{abstract}The solitary wave problem at the free surface of a two-dimensional, infinitely-deep and irrotational flow of water, under the influence of gravity, is formulated as a nonlinear pseudodifferential equation. A Pohozaev identity is used to show that it admits no solutions which asymptotically vanish faster than linearly.\end{abstract}

\medskip

%%%%%%%%%%%%%%%%%%%%%%%%%%%%%%%%%%%%%%%%%%%%%%%%%%%
%%%%%%%%%%%%%%%%%%%%%%%%%%%%%%%%%%%%%%%%%%%%%%%%%%%
%%%%%%%%%%%%%%%%%%%%%%%%%%%%%%%%%%%%%%%%%%%%%%%%%%%
\section{Introduction}\label{S:intro}
%%%%%%%%%%%%%%%%%%%%%%%%%%%%%%%%%%%%%%%%%%%%%%%%%%%
%%%%%%%%%%%%%%%%%%%%%%%%%%%%%%%%%%%%%%%%%%%%%%%%%%%
%%%%%%%%%%%%%%%%%%%%%%%%%%%%%%%%%%%%%%%%%%%%%%%%%%%

The purpose of this note is to relate steady waves at the free surface of a two-dimensional, infinitely-deep and irrotational flow of an incompressible inviscid fluid, acted upon by gravity, to the nonlinear pseudodifferential equation
\begin{equation}\label{E:deep}
\H w'=\mu(w+w\H w'+\H(ww')). \end{equation}
It will then be used to demonstrate the non-existence of solitary waves, for which the fluid surface asymptotically approaches a constant level over a nearly uniform flow. Here $w(x)$, $x\in\R$, measures the wave profile and $\mu=g/c^2$, where $g$ describes gravitational acceleration and $c$ denotes the speed of wave propagation; $\mu>0$ is physically realistic. The Hilbert transform is given as
\begin{equation}\label{D:H}
\H v(x)=\frac{1}{\pi} PV \int^\infty_{-\infty} \frac{v(t)}{x-t}~dt \qquad (x \in \R), \end{equation}
where $PV$ stands for the Cauchy principal value; alternatively $\widehat{\H v}(\xi)=-i\text{sgn}(\xi)\hat{v}(\xi)$ defines the operator in the Fourier space. Throughout $'=d/dx$.

\medskip

Solitary waves at the surface of water hold a central position in the theory of wave motion and they have historically stimulated a considerable part of its development, from Russell's famous horseback observations to the elucidation of the Korteweg-de Vries solitons. In the case where the flow depth is finite, the mathematical theory of solitary water waves dates back to constructions in \cite{FrHy} and \cite{Beale} of small-amplitude waves and it includes the global bifurcation result in \cite{AmTo81} and extensions in \cite{Hur08a} and \cite{GrWa08} to waves with vorticity. The symmetry and monotonicity properties were discussed in \cite{CrSt} and \cite{Hur08b} while the regularity properties were in \cite{Lewy} and \cite{Hur10}. Moreover non-uniqueness and linear instability were addressed in \cite{Plo} and in \cite{Lin-IMRN} for waves near the ``extremal" form. 

The present result, in stark contrast, states that no solitary water waves exist in the infinite-depth case, which, in case gravity acts downwards, asymptotically tend to the quiescent state faster than linearly.

\subsection*{Notation} 
The weighted H\"older space, written $C^{k+\alpha}_\rho(\R)$, for $k\geq 0$ an integer, $\alpha \in (0,1)$ and for $\rho\in \R$, is characterized via the norm
\begin{equation}\label{D:wHolder} 
\|v\|_{C^{k+\alpha}_\rho(\R)} =\sum_{j=0}^k\sup_{x\in\R}\left\langle x\right\rangle^{\rho}|v^{(j)}(x)|
+ \sup_{x\in\R}\sup_{|x-t|\leq 1}\left\langle x\right\rangle^{\rho} \frac{|v^{(k)}(x)-v^{(k)}(t)|}{|x-t|^\alpha},
\end{equation}
where $\left\langle x\right\rangle=(1+x^2)^{1/2}$. If $v \in C^{k+\alpha}_\rho(\R)$ then its derivatives of order up to $k$ vanish at least like $|x|^{-\rho}$ as $|x| \to \infty$.

%%%%%%%%%%%%%%%%%%%%%%%%%%%%%%%%%%%%%%%%%%%%%%%%%%%
%%%%%%%%%%%%%%%%%%%%%%%%%%%%%%%%%%%%%%%%%%%%%%%%%%%
\begin{theorem}\label{T:non-existence}
In the case of $\mu>0$, if $w \in H^1(\R)\cap C^{1+\alpha}_{1+\epsilon}(\R)$  satisfies \eqref{E:deep} for some $\alpha \in (0,1)$ and for some $\epsilon \in (0,1)$ then $w\equiv 0$. 

In the case of $\mu\leq0$, if $w \in H^1(\R)\cap C^{1+\alpha}(\R)$ satisfies \eqref{E:deep} for some $\alpha \in(0,1)$ and if in addition $\inf_{x\in\R}(1+\H w'(x))>0$ then $w\equiv 0$. \end{theorem}
%%%%%%%%%%%%%%%%%%%%%%%%%%%%%%%%%%%%%%%%%%%%%%%%%%%
%%%%%%%%%%%%%%%%%%%%%%%%%%%%%%%%%%%%%%%%%%%%%%%%%%%

Establishing an integral identity for the interface displacement, Sun in \cite{Sun97} argued the non-existence of solitary waves between two fluids flows of infinite extent, provided that the profile vanishes at infinity faster than linearly and that it is either purely elevated or purely depressed. Craig in \cite{Cra02} made an alternative proof based upon the maximum principle for the water wave problem in two and higher dimensions, without assuming decay of the fluid surface but requiring the positivity or negativity. Therefore a solitary wave on deep water cannot be everywhere positive (or negative). Do solitary water waves then exist, which oscillate about the mean fluid level? Theorem~\ref{T:non-existence} answers it negatively. If the effects of surface tension are factored into, on the other hand, solitary water waves were shown to arise in \cite{IK}, for instance, which necessarily change sign. 

\medskip

The solitary wave problem on deep water is introduced in Section~\ref{S:formulation} as a free boundary problem in potential theory. A ``regular" solution of the system, for which the velocity vanishes nowhere at the fluid surface (see Section \ref{SS:free-boundary}), then corresponds to a solution of \eqref{E:deep}, which serves as the basis of the present account, and vice versa so long as the fluid surface lacks self-intersections and cusps. Exploiting a conformal mapping of the fluid domain, the proof resembles that for Babenko's equation in the Stokes (periodic) wave setting, which already led to much progress in \cite{BDT00a} and \cite{BDT00b} among others. 

Observe that \eqref{E:deep} enjoys the scaling symmetry under 
\[w(x)\mapsto\lambda^{-1}w(\lambda x)\quad\text{and}\quad\mu\mapsto\lambda\mu\] 
for any $\lambda>0$ and that the vector field $xd/dx$ generating the symmetry commutes with the Hilbert transform (see \eqref{E:Hcomm}). They offers in Lemma~\ref{L:non-existence} an extremely simple non-ex\-is\-tence proof for \eqref{E:deep} whenever $xw'\in L^2(\R)$, which is reminiscent of Pohozaev identities techniques. For  differential equations one may further employ a truncation argument to promote $L^2$-based spaces with weight to locally $L^\infty$-based spaces; see \cite{dBS}, for instance, for the Kadomtsev-Petviashivilli equation. But the Hilbert transform, being pseudodifferential, does not commute with functions as a rule. We instead make an effort to understand the asymptotic behavior of the solution to refine the conclusion.

In the case of $\mu<0$, i.e. gravity acts oppositely to what is physically realistic, the linearized operator of \eqref{E:deep} becomes $\H d/dx-\mu$ plus a function after the Plotnikov transformation (see \eqref{E:plotnikov}). Accordingly $\mu$ may be regarded as an eigenvalue of $\H d/dx=(-d^2/d x^2)^{1/2}$ with potential. Spectral properties in \cite{CaMaSi}, for instance, of a relativistic Schr\"odinger operator then reveals in Lemma~\ref{L:decay-} that the derivative of a solution to \eqref{E:deep} decays quadratically. In the Stokes wave setting, incidentally, a non-existence proof based upon duality is found in \cite{Tol02}.

In case gravity acts downwards, on the other hand, $\mu>0$ is contained in the essential spectrum of $\H d/dx$ and hence the proof of Lemma~\ref{L:decay-} is not applicable. To attain decay nevertheless, we shall impose a solvability condition. Specifically, if a solution of \eqref{E:deep} is assumed to decay faster than $x^{-1}$ as $|x|\to \infty$ then it does like $x^{-2}$ upon a bootstrapping argument in \cite{CrSt} or \cite{Sun97}, for instance. The author has not yet succeeded in removing the extra condition although it is desirable. For a broad class of interfacial fluids problems in the infinite-depth case, including the Benjamin-Ono equation and the water wave problem with surface tension (see \cite{IK}, for instance), to compare, solitary waves vanish quadratically, suggesting that the decay rate of $x^{-1-}$ is not outrageous. 

\medskip

Concerning the non-stationary water wave problem in like setting, the associated linear operator $\partial^2/\partial t^2+\H\partial/\partial x$ remains invariant under the vector field $\frac12t\partial/\partial t+x\partial/\partial x$, where $t\in\R$ denotes the temporal variable, analogously to that $\H d/dx$ is invariant under $xd/dx$ in the steady wave problem \eqref{E:deep}. As a matter of fact dispersive estimates plus a profound understanding of the nonlinearity shed light in \cite{Wu3} to the almost-global existence for small data. To interpret, solitary waves of small amplitude are unlikely to arise at the surface of a two-dimensional infinitely-deep flow of water, if the profile is sufficiently smooth and if in addition it vanishes sufficiently fast at infinity. Theorem \ref{T:non-existence} bears it out. 

The development in Section \ref{S:formulation} of \eqref{E:deep} relies upon conformal mapping techniques. It seems not to work in higher dimensions, leaving open the tantalizing possibility of steady water waves in three dimensions, which oscillate yet become flat at infinity.  

%%%%%%%%%%%%%%%%%%%%%%%%%%%%%%%%%%%%%%%%%%%%%%%%%%%
%%%%%%%%%%%%%%%%%%%%%%%%%%%%%%%%%%%%%%%%%%%%%%%%%%%
%%%%%%%%%%%%%%%%%%%%%%%%%%%%%%%%%%%%%%%%%%%%%%%%%%%
\section{Formulation}\label{S:formulation}
%%%%%%%%%%%%%%%%%%%%%%%%%%%%%%%%%%%%%%%%%%%%%%%%%%%
%%%%%%%%%%%%%%%%%%%%%%%%%%%%%%%%%%%%%%%%%%%%%%%%%%%
%%%%%%%%%%%%%%%%%%%%%%%%%%%%%%%%%%%%%%%%%%%%%%%%%%%
The solitary wave problem on deep water is formulated as the nonlinear pseudodifferential equation \eqref{E:deep}.

%%%%%%%%%%%%%%%%%%%%%%%%%%%%%%%%%%%%%%%%%%%%%%%%%%%
%%%%%%%%%%%%%%%%%%%%%%%%%%%%%%%%%%%%%%%%%%%%%%%%%%%
\subsection{The free boundary problem}\label{SS:free-boundary}
%%%%%%%%%%%%%%%%%%%%%%%%%%%%%%%%%%%%%%%%%%%%%%%%%%%
%%%%%%%%%%%%%%%%%%%%%%%%%%%%%%%%%%%%%%%%%%%%%%%%%%%
The steady water wave problem in the simplest form concerns a two-dimensional, infinitely-deep and irrotational flow of an incompressible inviscid fluid and wave motions at the surface layer, under gravity.  By steady we mean that the flow as well as the surface wave move at a constant speed from right to left without changing their configuration. The effects of surface tension are neglected. The solitary wave problem then seeks for solutions, for which the fluid surface asymptotically tends to a constant level and the flow in the far field is nearly uniform.

In the dimensionless coordinates moving at the speed of wave propagation, let the parametric curve
\[\Gamma=\{ (u(x), v(x)): x \in \R \}\]
represent the fluid surface. Assume that
\begin{subequations}\label{E:stream}
\begin{align}
& \text{$x \mapsto (u(x),v(x))$ is continuously differentiable}; \label{E:gamma-injective}\\
& \text{$u'(x)>0$ and $u'(x)^2+v'(x)^2<\infty$ for every $x \in \R$}; \label{E:gamma-regular}\\
& \text{$(u(x)-x,v(x)) \to 0$ as $|x| \to \infty$.}\label{E:gamma-infty}
\end{align}
Let $\Omega$ denote the open region in the $(X,Y)$-plane below $\Gamma$, occupied by the fluid. {\em The solitary wave problem on deep water} is to find a curve $\Gamma$ satisfying \eqref{E:gamma-injective}-\eqref{E:gamma-infty}, a parameter $\mu \in \R$ and a function $\psi$ defined over $\overline{\Omega}$ such that 
\begin{align}
& \text{$\psi \in C^1(\overline{\Omega})\cap C^2(\Omega)$ and 
$\psi$ is harmonic in $\Omega$;} \label{E:laplace} \\
& \text{$\psi=0$ on $\Gamma$;} \label{E:kinematics-surface}\\
& \text{$\psi(X,Y) -Y \to 0$ as $|(X,Y)| \to \infty$;} \label{E:psi-infty}\\
& \text{$\lim_{(X,Y)\to(u(x),v(x))} |\nabla \psi(X,Y)|^2
+2\mu v(x) = 1$ for every $x \in \R$.} \label{E:bernoulli-weak}
\end{align}\end{subequations}
Details are discussed in \cite{Benjamin}, for instance. Here we merely hit the main points. 

\medskip

The kinematic boundary condition \eqref{E:kinematics-surface} states that the fluid surface itself makes a streamline while the dynamic boundary condition \eqref{E:bernoulli-weak} manifests Bernoulli's law of constant pressure at the surface. They determine the free boundary $\Gamma$, which by \eqref{E:gamma-injective} and \eqref{E:gamma-regular} is the graph of a $C^1$ function that is nowhere vertical. The (non-di\-men\-sion\-alized) stream function $\psi$ is to describe the fluid motion. Specifically $(-\psi_Y, \psi_X)$ denotes the steady velocity field. Boundary conditions \eqref{E:gamma-infty} and \eqref{E:psi-infty} mean, respectively, that the fluid surface becomes asymptotically horizontal and that the flow in the far field is nearly uniform. 

\medskip

A solution triple $\Gamma$, $\mu$ and $\psi$ of \eqref{E:stream} is said {\em regular} if $|\nabla \psi|> 0$ on $\Gamma$. A theorem of Lewy in \cite{Lewy} results in that the fluid surface of a regular solution to \eqref{E:stream} is real analytic. In the finite-depth case, on the other hand, $\nabla \psi=0$ at the crest of a wave of the maximum height, namely an ``extremal" wave; see \cite{AmTo81}, for instance.

%%%%%%%%%%%%%%%%%%%%%%%%%%%%%%%%%%%%%%%%%%%%%%%%%%%
%%%%%%%%%%%%%%%%%%%%%%%%%%%%%%%%%%%%%%%%%%%%%%%%%%%
\subsection{Complex Hardy spaces in the half plane}\label{SS:preliminaries}
%%%%%%%%%%%%%%%%%%%%%%%%%%%%%%%%%%%%%%%%%%%%%%%%%%%
%%%%%%%%%%%%%%%%%%%%%%%%%%%%%%%%%%%%%%%%%%%%%%%%%%%
Throughout \[\mathbb{P}=\{x+iy\in\C:y <0\}\] denotes the open lower half plane in the complex field. 

The following definitions and discussion are taken from \cite[Chapter I and Chapter II]{Gar} and \cite[Chapter 11]{Duren} among others.  

\medskip

The complex Hardy space in the lower half plane, written $\mathbf{H}^p(\mathbb{P})$, for $p \in(0,\infty]$, consists of holomorphic functions $f:\mathbb{P} \to \C$ such that 
\[ \|f\|_{\mathbf{H}^p(\mathbb{P})}=\sup_{y \in (-\infty,0)}  \| f(\cdot+iy)\|_{L^p(\R)} <\infty. \]
If $f \in \mathbf{H}^p(\P)$, $p\in (0,\infty)$, then its non-tangential limit 
\begin{equation}\label{E:f*} 
f^*(x):=\lim_{y \to 0-} f(x+iy) \end{equation} 
exists in $L^p(\R)$ and for almost all $x \in\R$. In the case of $f \in \mathbf{H}^\infty(\P)$ correspondingly $f^*$ exists almost everywhere on $\R$; see \cite[Theorem 11.1]{Duren} or \cite[Chapter II, Theorem 3.1]{Gar}, for instance. If $f \in \mathbf{H}^p(\P)$, $p \in (0,\infty]$, then either $f\equiv 0$ or ${\displaystyle \int^\infty_{-\infty} \frac{|\log|f^*(x)||}{1+x^2}~dx <\infty}$. If in addition $f^*=0$ on a set of positive measure then $f\equiv 0$.  

If $f \in \mathbf{H}^p(\mathbb{P})$, $p\in [1,\infty]$, then \cite[Theorem 11.2]{Duren} or \cite[Chapter II, Corollary 3.2]{Gar}, for instance, state that
\[ f(x+iy)=\frac{1}{\pi}\int^\infty_{-\infty} \frac{y}{(x-t)^2+y^2}f^*(t)~dt.\]
That is to say, $f$ agrees with the Poisson integral of its boundary function. Conversely, if $v \in L^p(\R)$ and if ${\displaystyle f(x+iy)=\frac{1}{\pi}\int^\infty_{-\infty} \frac{y}{(x-t)^2+y^2}v(t)~dt}$ is holomorphic in $\mathbb{P}$ then $f \in\mathbf{H}^p(\mathbb{P})$. Moreover $f^*=v$ almost everywhere on $\R$. 

\medskip

If $u$ is harmonic in $\P$ and if $\sup_{y \in (-\infty,0)}\|u(\cdot+iy)\|_{L^p(\R)} <\infty$, $p \in (1,\infty]$, then the classical Fatou lemma (see \cite[Chapter~I, Theorem~5.3]{Gar}, for instance) implies that its non-tangential limit $u^*$ exists almost everywhere on $\R$ and $u$ is the Poisson integral of $u^*$. If in addition $u^*$ is uniformly continuous on $\R$ then $u \to u^*$ uniformly as $y \to 0-$; see \cite[Chapter~I, Theorem~3.1]{Gar}, for instance.

\medskip

In the case of $p=2$, in particular, the Paley-Wiener theorem (see \cite[The\-o\-rem 11.9]{Duren}, for instance) implies that if $v \in L^2(\R)$ then there exists a unique function $\mathcal{R}v \in \mathbf{H}^2(\mathbb{P})$ such that 
\begin{equation}\label{E:R} (\mathcal{R}v)^*=\H v+iv;\end{equation} 
recall that $\H$ is the Hilbert transform (see \eqref{D:H}). Conversely, if $f \in \mathbf{H}^2(\mathbb{P})$ then $f^*=\H v+iv$ for some $v \in L^2(\R)$. 

Elementary properties of the Hilbert transform include that $\H: L^p(\R) \to L^p(\R)$, $p\in (1,\infty)$, is bounded and that the adjoint of $\H$ is $-\H$. Moreover the Privalov theorem states that $\H:C^\alpha(\R) \to C^\alpha(\R)$, $\alpha\in (0,1)$, is bounded.

\medskip

Lastly, if $v \in L^\infty(\R)$, $v \neq 0$ almost everywhere on $\R$ and if $1/v \in L^\infty(\R)$ then there exists $f \in \mathbf{H}^\infty(\P)$, called an {\em outer function} of $v$, such that $|f^*|=v$ almost everywhere on $\R$ if and only if ${\displaystyle \int^\infty_{-\infty}\frac{|\log|v(x)||}{1+x^2}~dx <\infty}$. Moreover $1/f \in \mathbf{H}^\infty(\P)$. We refer the reader to \cite[Chapter~II, Theorem~4.4]{Gar}, for instance.

%%%%%%%%%%%%%%%%%%%%%%%%%%%%%%%%%%%%%%%%%%%%%%%%%%%
%%%%%%%%%%%%%%%%%%%%%%%%%%%%%%%%%%%%%%%%%%%%%%%%%%%
\subsection{Reduction to a single equation}\label{SS:equation}
%%%%%%%%%%%%%%%%%%%%%%%%%%%%%%%%%%%%%%%%%%%%%%%%%%%
%%%%%%%%%%%%%%%%%%%%%%%%%%%%%%%%%%%%%%%%%%%%%%%%%%%
This subsection concerns the equivalence between the system \eqref{E:stream} and the single nonlinear pseudodifferential equation \eqref{E:deep}. 

In what follows the two-dimensional space is identified with the complex plane whenever it is convenient to do so. 

%%%%%%%%%%%%%%%%%%%%%%%%%%%%%%%%%%%%%%%%%%%%%%%%%%%
%%%%%%%%%%%%%%%%%%%%%%%%%%%%%%%%%%%%%%%%%%%%%%%%%%%
\begin{lemma}\label{P:stream->bernoulli}
Suppose that a curve $\Gamma$ and a function $\psi$ defined in the plane below $\Gamma$ form a regular solution of \eqref{E:stream} for some $\mu \in \R$. Let  $-\phi$ be a harmonic conjugate of $\psi$. 
Then $\phi+i\psi$ is a conformal bijection from $\overline{\Omega}$ onto $\overline{\mathbb{P}}$ which maps $\infty$ to $\infty$.

Let $w=\emph{Im}\,(\phi+i\psi)^{-1}(\cdot+i0)$. If $w \in H^1(\R)$ then
\begin{equation}\label{E:bernoulli}
(1-2\mu w)((1+\H w')^2+(w')^2)=1\end{equation} 
and 
\begin{equation}\label{E:injective}
\inf_{x\in\R}(1+\H w'(x))>0.\end{equation}
\end{lemma}
%%%%%%%%%%%%%%%%%%%%%%%%%%%%%%%%%%%%%%%%%%%%%%%%%%%
%%%%%%%%%%%%%%%%%%%%%%%%%%%%%%%%%%%%%%%%%%%%%%%%%%%

\begin{proof}
Since $|\nabla \psi|>0$ on $\Gamma$ by hypothesis, a theorem of Lewy in \cite{Lewy} ensures that $\Gamma$ is a real analytic curve. By \eqref{E:gamma-regular}, moreover, $\Gamma$ is the graph of a function that is nowhere vertical. 

Note from \eqref{E:kinematics-surface}, \eqref{E:gamma-infty} and \eqref{E:psi-infty} that the harmonic function $\psi$ is negative in the smooth domain $\Omega$ and attains the maximum over $\overline{\Omega}$ at every point of $\Gamma$. By the Hopf boundary point lemma, therefore, $\phi_X=\psi_Y>0$ at $\Gamma$. Moreover $\psi_Y$ is harmonic in $\Omega$ and $\psi_Y(X,Y) \to 1$ as $|(X,Y)| \to \infty$ (see \eqref{E:psi-infty}). The maximum principle then implies that $\phi_X=\psi_Y>0$ throughout $\overline{\Omega}$. Consequently $\phi+i\psi$ makes a conformal bijection from $\overline{\Omega}$ onto $  \overline{\P}$. Furthermore a streamline $\{(X,Y) \in \Omega: \psi(X,Y)=-\psi_0\}$ for each $\psi_0 \in [0,\infty)$ is the graph of a smooth function on $\R$. 

\medskip

Let $F:\mathbb{P}\to \Omega$ be a conformal bijection such that $F(\infty)=\infty$. Then $F$ extends homeomorphically from $\overline{\P}$ onto $\overline{\Omega}$ by a famous theorem of Carath\'eodory (see \cite[pp. 89]{Gar}, for instance). Since $\psi \circ F$ is harmonic in $\mathbb{P}$, continuous on $\overline{\mathbb{P}}$ and nonpositive, moreover, \cite[Chapter I, Theorem 3.5]{Gar}, for instance, dictates that there exist a constant $\alpha\geq 0$ and a Borel measure $\sigma \geq 0$ on $\R$ such that 
\begin{equation}\label{E:sigma} 
\int^\infty_{-\infty} \frac{d\sigma(t)}{1+t^2}<\infty \quad \text{and}\quad
\psi\circ F(x+iy)=\alpha y+\frac{1}{\pi}\int^\infty_{-\infty} \frac{y}{(x-t)^2+y^2}~d\sigma(t)\end{equation}
for $x+iy\in\mathbb{P}$. We claim that $\sigma \equiv 0$. 

Let $v$ be a continuous function on $\R$ with bounded support, say $B \subset \mathbb{R}$. Clearly 
\[ \int^\infty_{-\infty} \left| \frac{v(x)y}{(x-t)^2+y^2}\right|dx \leq \frac{C}{1+t^2}\] 
for some constant $C>0$ uniformly for $y$ near $0$. Since $\psi \circ F(x+iy) \to 0$ as $y \to 0-$ uniformly for $x$ in any bounded subset of $\R$ (see \eqref{E:kinematics-surface}), moreover,
\[ \int^\infty_{-\infty} v(x) \psi \circ F(x+iy)~dx \to 0 \qquad \text{as $y \to 0-$}.\] 
In view of \eqref{E:sigma}, therefore, Fubini's theorem and the dominated convergence theorem yield that
\begin{align*}
0=&\lim_{y\to 0-} \int^\infty_{-\infty} v(x) \frac{1}{\pi}\int^\infty_{-\infty} \frac{y}{(x-t)^2+y^2}~d\sigma(t)dx \\
=&\lim_{y \to 0-} \frac{1}{\pi}\int^\infty_{-\infty} \int^\infty_{-\infty} \frac{v(x)y}{(x-t)^2+y^2}~dxd\sigma(t)
=\int^\infty_{-\infty} v(t)~d\sigma(t). \end{align*}
Since $v$ is arbitrary, this proves the claim. 

To recapitulate, $(\psi \circ F)(x+iy)=\alpha y$ for some $\alpha\geq 0$. Since $\psi \not\equiv 0$ and since $(\phi+i\psi)\circ F$ is holomorphic in $\P$, furthermore, $\alpha> 0$ and 
\[((\phi+i\psi) \circ F)(x+iy)=\alpha(x+iy) +\beta\quad\text{for some $\beta \in \R$.}\]  
To conclude, $F^{-1}=\alpha^{-1}(\phi+i\psi-\beta)$ and, in turn, $\phi+i\psi$ are conformal bijections from $\overline{\Omega}$ onto $\overline{\P}$ which map $\infty$ to $\infty$. Below let $Z=X+iY$ denote the complex coordinate in $\overline{\Omega}$ and we use $z=x+iy$ for $\overline{\P}$ such that $z=(\phi+i\psi)(Z)$. 

\medskip

It remains to show \eqref{E:bernoulli} and \eqref{E:injective}. Let 
\[\zeta:=(\phi+i\psi)^{-1}:\overline{\P} \to \overline{\Omega}.\] 
Then it is a conformal bijection such that $\zeta(\infty)=\infty$. Furthermore $0<|\zeta'|<\infty$ throughout $\overline{\P}$. Indeed, since $0<|\nabla \psi|<\infty$ throughout $\overline{\Omega}$ by the maximum principle and by hypothesis and \eqref{E:psi-infty},
\begin{equation}\label{E:zeta'}\zeta'(z)=(\phi+i\psi)' (Z)^{-1}=(\psi_Y-i\psi_X)(Z)^{-1}\end{equation}
is well-defined for each $z \in \overline{\P}$, where $Z=\zeta(z) \in \overline{\Omega}$. Consequently $\overline{\Omega}$ may be parametrized by $x+iy \in \overline{\P}$ and $\Gamma$ may be parametrized as $x \mapsto \zeta(x+i0)$. 

Let $w(x)=\text{Im}\,\zeta(x+i0)$, $x\in\R$. We claim that
\begin{equation}\label{E:zeta}
\zeta(z)=z+\mathcal{R}w(z)+\alpha\quad \text{for every $z \in \overline{\P}$}
\end{equation} for some $\alpha \in \R$, 
where $\mathcal{R}w$ is given in \eqref{E:R} and well-defined thanks to the regularity hypothesis. 

Note that $\text{Im}\,\zeta'$ is harmonic in $\P$ and continuous on $\overline{\P}$. Since it is bounded throughout $\overline{\P}$ (see \eqref{E:zeta} and \eqref{E:psi-infty}), moreover, $\text{Im}\, \zeta'(\cdot +i0)=w'$ by \cite[Chapter~I, Theorem~3.1]{Gar}, for instance; see Section~\ref{SS:preliminaries}. Since $w' \in L^2(\R)$ by hypothesis, on the other hand, the Paley-Wiener theorem (see \cite[Theorem~11.9]{Duren}, for instance) ensures that $\mathcal{R}w' \in \mathbf{H}^2({\P})$ exists such that $(\mathcal{R}w')^*=\H w'+iw'$ in $L^2(\R)$; recall from \eqref{E:f*} that $(\mathcal{R}w')^*$ is the boundary function of $\mathcal{R}w'$. In addition \cite[The\-o\-rem~11.2]{Duren}, for instance, dictates that $\mathcal{R}w'$ agrees with the Poisson integral of its boundary function; see Section~\ref{SS:preliminaries}. Therefore $\text{Im}\,\mathcal{R}w'$ is the Poisson integral of $w'$. Since $w'$ is smooth on $\R$ in light of a Lewy theorem in \cite{Lewy}, furthermore, the Fatou lemma (see \cite[Chapter~I, Theorem~5.3]{Gar}, for instance) asserts that $\text{Im}\,\mathcal{R}w' \to w'$ everywhere on $\R$, whence $\text{Im}\,\zeta'\equiv \text{Im}\,\mathcal{R}w'$ in $\overline{\P}$. The Cauchy-Riemann equations then yield that
\[ \zeta'=\mathcal{R}w'+\beta \quad \text{in $\overline{\P}$ for some $\beta\in\R$.}\] 
Since $\zeta'(z)\to 1$ by \eqref{E:psi-infty}, moreover, $\beta=1$. An integration therefore reveals \eqref{E:zeta} for some $\alpha \in \C$. But $\alpha \in \R$ since $\text{Im}\,\zeta(z)-y\to 0$ as $|z| \to \infty$ by \eqref{E:psi-infty}. This proves the claim. 

To summarize, \[ \Gamma=\{(x+\H w(x) +\alpha,w(x)):x \in \R\}\] 
and the former condition in \eqref{E:gamma-regular} is written as \eqref{E:injective}. Upon evaluating \eqref{E:zeta'} at $z$ on the real axis, correspondingly, \eqref{E:bernoulli-weak} becomes
\[ (1+\H w'(x))^2+w'(x)^2=|\zeta'(x)|^2=|\nabla \psi(Z(x+i0))|^{-2}=(1-2\mu w(x))^{-1}\] 
for every $x \in \R$. Incidentally $1-2\mu w>0$ everywhere on $\R$.\end{proof}

Notice that \eqref{E:bernoulli} recasts Beroulli's law of constant pressure at the fluid surface (see \eqref{E:bernoulli-weak}) while \eqref{E:injective} states that the surface is globally injective. 
%Next is to show the one-to-one correspondence between \eqref{E:deep} and \eqref{E:bernoulli}.

%%%%%%%%%%%%%%%%%%%%%%%%%%%%%%%%%%%%%%%%%%%%%%%%%%%%
%%%%%%%%%%%%%%%%%%%%%%%%%%%%%%%%%%%%%%%%%%%%%%%%%%%%
\begin{lemma}\label{P:main<->bernoulli}
{\rm (a)} If $w \in H^1(\R)$ satisfies \eqref{E:deep} and if $w' \in L^\infty(\R)$ then $w$ satisfies \eqref{E:bernoulli}. 

{\rm (b)} If $w\in H^1(\R)$ satisfies \eqref{E:bernoulli} and \eqref{E:injective} and if $w'$, $\H w' \in C(\R)$ then $w$ satisfies \eqref{E:deep}.
\end{lemma}
%%%%%%%%%%%%%%%%%%%%%%%%%%%%%%%%%%%%%%%%%%%%%%%%%%%%
%%%%%%%%%%%%%%%%%%%%%%%%%%%%%%%%%%%%%%%%%%%%%%%%%%%%

\begin{proof} 
(a) The proof is similar to that in the first part of \cite[Theorem~2.3]{BDT00a}. Hence we merely sketch the detail. 

Suppose that $w$ satisfies \eqref{E:deep} with the stated regularity. Let $u=w-\mu w^2$ and we rewrite \eqref{E:deep} as  
\begin{equation}\label{E:gradJ} 
(1-2\mu w)(1+\H w')+\H u'=1. \end{equation} 
Since $u' \in L^2(\R)$ by a Sobolev inequality, the Paley-Wiener theorem (see \cite[The\-o\-rem 11.9]{Duren}, for instance) ensures that $\mathcal{R}u' \in\mathbf{H}^2(\mathbb{P})$ exists such that $(\mathcal{R}u')^*=\H u'+iu'$ in $L^2(\R)$; recall from \eqref{E:f*} that $(\mathcal{R}u')^*$ is the non-tangential limit of $\mathcal{R}u'$. Correspondingly $\mathcal{R}w' \in \mathbf{H}^2(\mathbb{P})$ exists such that $(\mathcal{R}w')^*=\H w'+iw'$ in $L^2(\R)$. 

Let \[ V=\text{Im}\left(\left(1-\mathcal{R}u'\right) (1+\mathcal{R}w')\right)
=\text{Im}(-\mathcal{R}u'+\mathcal{R}w'-\mathcal{R}u'\mathcal{R}w').\] 
Then $V$ is harmonic in $\mathbb{P}$ and $V(\cdot+iy) \to V(\cdot+i0)$ in $L^2(\R)$ as $y \to 0-$. Indeed, since $(\H u'+iu')(\H w'+iw')=(\H u'\H w'-u'w')+i(u'\H w'+ w'\H u')$ and since $u'\H w'+ w'\H u' \in L^2(\R)$, the Cauchy integral formula furnishes that
\[\H(u'\H w'+w'\H u')=\H u'\H w'-u'w' \in L^2(\R)\] 
(alternatively an explicit calculation in the Fourier space reveals the identity);
accordingly $\mathcal{R}u'\mathcal{R}w'=\mathcal{R}(u'\H w'+w'\H u') \in \mathbf{H}^2(\mathbb{P})$ by \cite[Chapter~II, Corollary~4.2]{Gar}, for instance. Evaluating $V$ on the real axis, we use \eqref{E:gradJ} to obtain that
\begin{align*} 
V(\cdot+0i)=&\text{Im}\left( (1-\H u'-iu')(1+\H w'+iw')\right) \\
%=&\text{Im}\left((1-2\mu w)(1+\H w'-iw')(1+\H w'+iw')\right) \\
=&\text{Im}\left((1-2\mu w)((1+\H w')^2+(w')^2)\right)=0
\end{align*}
for almost all $x \in \R$. The Poisson integral technique then enforces that $V\equiv 0$ in $\mathbb{P}$.
In view of the Cauchy-Riemann equations, consequently, $\text{Re}((1-\mathcal{R}u') (1+\mathcal{R}w'))$
is constant in $\mathbb{P}$. Since $(\mathcal{R}u)^*(x), (\mathcal{R}w')^*(x) \to 0$ as $|x| \to \infty$, 
moreover, the constant must be $1$ and
\begin{align*} 
1=\text{Re}\left(\left(1-\mathcal{R}u\right) (1+\mathcal{R}w')\right)= (1-2\mu w)((1+\H w')^2+(w')^2)
\end{align*}
almost everywhere on $\R$. 

\medskip

(b) Suppose that $w\in H^1(\R)$ satisfies \eqref{E:bernoulli}, \eqref{E:injective} and enjoys the stated properties. Let $W=1+\mathcal{R}w'$. Then 
\begin{equation}\label{E:W*}
W^*=1+\H w'+iw'\end{equation} 
everywhere on $\R$. Since $0<(1+\H w')^2+(w')^2<\infty$ everywhere on $\R$ by hypothesis, $W \in C^1(\overline{\P}) \cap \mathbf{H}^\infty(\P)$. Moreover, since $\text{Re}\,W^*=1+\H w'$ is positive everywhere on $\R$ by hypothesis, the maximum principle implies that $W$ is nowhere zero on $\overline{\P}$. Since ${\displaystyle \int^\infty_{-\infty} \frac{|\log |W^*(x)||}{1+x^2}~dx<\infty}$, furthermore, \cite[Chapter~II, Theorem~4.6]{Gar}, for instance, asserts that $W$ is an outer function; see at the end of Section~\ref{SS:preliminaries}. Therefore $1/W \in \mathbf{H}^\infty(\mathbb{P})$.

Rewriting \eqref{E:bernoulli} as \[|W^*|^2=1-2\mu w,\] 
since $1/W^*-1=(1-2\mu w)\overline{W^*}-1 \in L^2(\R)$ by \eqref{E:W*} and by a Sobolev inequality, $1/W-1 \in \mathbf{H}^2(\mathbb{P})$ in light of \cite[Theorem~11.2]{Duren}, for instance; see Section~\ref{SS:preliminaries}. The Paley-Wiener theorem (see \cite[Theorem~11.9]{Duren}, for instance) then ensures that $1/W^*-1=\H u+iu$ for some $u \in L^2(\R)$. At last a straightforward calculation reveals \eqref{E:deep} since
\[ 1+\H u=(1-2\mu w)(1+\H w') \quad\text{and}\quad u=-(1-2\mu w) w'.\]
\end{proof}

\begin{corollary}\label{C:main<->bernoulli}
Suppose that $w \in H^1(\R)\cap C^{1+\alpha}(\R)$, $\alpha \in (0,1)$, satisfies \eqref{E:injective}. Then, $w$ satisfies \eqref{E:deep} if and only if it satisfies \eqref{E:bernoulli}.
\end{corollary}

\begin{proof}Since $w', \H w' \in C^\alpha(\R)$ by Privalov's theorem for the Hilbert transform (see Section~\ref{SS:preliminaries}), \eqref{E:bernoulli} ensures that $0<(1+\H w')^2+(w')^2<\infty$  everywhere on $\R$.\end{proof}

We state the main result on equivalence.

%%%%%%%%%%%%%%%%%%%%%%%%%%%%%%%%%%%%%%%%%%%%%%%%%%%%
%%%%%%%%%%%%%%%%%%%%%%%%%%%%%%%%%%%%%%%%%%%%%%%%%%%%
\begin{proposition}\label{T:equivalence}
Suppose that a curve $\Gamma$ and a function $\psi$ defined in the plane below $\Gamma$ form a regular solution of \eqref{E:stream} for some $\mu \in \R$. If $w \in H^1(\R) \cap C^{1+\alpha}(\R)$, $\alpha \in (0,1)$, where $w$ is in Proposition~\ref{P:stream->bernoulli}, then $w$ satisfies \eqref{E:deep} in addition to the conclusions of Lemma~\ref{P:stream->bernoulli}. 

Conversely, suppose that $w\in H^1(\R) \cap C^{1+\alpha}(\R)$, $\alpha \in (0,1)$, satisfies \eqref{E:deep} 
for some $\mu \in \R$ as well as \eqref{E:injective}. Let $\Gamma=\{(x+\H w(x), w(x)): x\in \R\}$ and let $\Omega$ be the open domain below $\Gamma$. Then, there exists a conformal bijection $\Phi: \Omega \to \mathbb{P}$ such that $\Gamma$, $\emph{Im}\,\Phi$ and $\mu$ form a regular solution of \eqref{E:stream}.
\end{proposition}
%%%%%%%%%%%%%%%%%%%%%%%%%%%%%%%%%%%%%%%%%%%%%%%%%%%
%%%%%%%%%%%%%%%%%%%%%%%%%%%%%%%%%%%%%%%%%%%%%%%%%%%

\begin{proof}
If $\Gamma$ and $\psi$, $\mu$ form a regular solution of \eqref{E:stream} with the stated properties then the former assertion follows from Lemma~\ref{P:stream->bernoulli} and  Corollary~\ref{C:main<->bernoulli}. 

The proof of the converse is nearly identical to that in the second part of \cite[Theorem 2.3]{BDT00a}. Hence we omit the detail. \end{proof}

%%%%%%%%%%%%%%%%%%%%%%%%%%%%%%%%%%%%%%%%%%%%%%%%%%%
%%%%%%%%%%%%%%%%%%%%%%%%%%%%%%%%%%%%%%%%%%%%%%%%%%%
%%%%%%%%%%%%%%%%%%%%%%%%%%%%%%%%%%%%%%%%%%%%%%%%%%%
\section{Non-existence}\label{S:proof}
%%%%%%%%%%%%%%%%%%%%%%%%%%%%%%%%%%%%%%%%%%%%%%%%%%%
%%%%%%%%%%%%%%%%%%%%%%%%%%%%%%%%%%%%%%%%%%%%%%%%%%%
%%%%%%%%%%%%%%%%%%%%%%%%%%%%%%%%%%%%%%%%%%%%%%%%%%%
Attention is turned to the proof of Theorem \ref{T:non-existence}. 

\begin{lemma}\label{L:non-existence}
If $w \in H^1(\R)$ satisfies \eqref{E:deep} for some $\mu \in\R$ and if $xw' \in L^2(\R)$ then $w\equiv 0$. \end{lemma}

\begin{proof}
Multiplying \eqref{E:deep} by $xw'$ and integrating over $\R$ we use integration by parts and that the adjoint of $\H$ is $-\H$ to obtain that
\begin{align}
\int^\infty_{-\infty} xw' \H w'~dx =&\mu\int^\infty_{-\infty} xww'~dx 
+ \mu\int^\infty_{-\infty}  xww'\H w'~dx +\mu\int^\infty_{-\infty} xw' \H(ww')~dx  \notag \\
=&-\frac{\mu}{2}\int^\infty_{-\infty} w^2~dx
-\frac{\mu}{2}\int^\infty_{-\infty} (w^2)'[\H,x]w'~dx, \label{E:pohozaev}
\end{align}
where $[\,,]$ means the commutator. Note that all integrals make sense by hypotheses. In particular $w\H w', \H(ww') \in L^2(\R)$.  

A straightforward calculation (see \eqref{D:H}) on the other hand reveals that
\begin{equation}\label{E:Hcomm} 
[\H, xd/dx]w(x)=[\H,x]w'(x)=\frac{1}{\pi} \int^\infty_{-\infty} w'(x)~dx\equiv 0.\end{equation}
That is to say, $xd/dx$ commutes with the Hilbert transform. Accordingly the second term on the right side of \eqref{E:pohozaev} vanishes. Since
\[ \int^\infty_{-\infty} x \, w'\H w' dx 
=-\int^\infty_{-\infty} x(\H w') w' dx -\int^\infty_{-\infty} ([\H,x]w')w'~dx\equiv 0,\]
moreover, \eqref{E:pohozaev} reduces to that ${\displaystyle \int^\infty_{-\infty} w^2~dx =0}$. \end{proof}

The proof is reminiscent of Pohozaev identities techniques. For differential equations (see \cite{dBS}, for instance) one may further combine it with a truncation argument to promote the conclusion to the $L^\infty_{loc}$-space setting; to compare, Lemma~\ref{L:non-existence} is in the $L^2$-space setting with weight. But the Hilbert transform unfortunately does not commute with functions as a rule. 

We instead make an effort to understand the asymptotic behavior of the solution. Note that if $xw'\in L^2(\R)$ and if $w$ vanishes algebraically at infinity then, necessarily, $w'(x) \to 0$ faster than $|x|^{-3/2}$ as $|x| \to \infty$.

%%%%%%%%%%%%%%%%%%%%%%%%%%%%%%%%%%%%%%%%%%%%%%%%%%%
%%%%%%%%%%%%%%%%%%%%%%%%%%%%%%%%%%%%%%%%%%%%%%%%%%%
\begin{lemma}\label{L:decay-}
If $w\in H^1(\R)\cap C^{1+\alpha}(\R)$, $\alpha\in (0,1)$, satisfies \eqref{E:deep} for some $\mu<0$ as well as \eqref{E:injective} then $w'(x)\leq C(1+x^2)^{-1}$ for every $x \in \R$ for some constant $C>0$. 
\end{lemma}

In light of Lemma~\ref{L:non-existence} it proves the second part of Theorem~\ref{T:non-existence}. 

Colloquially speaking, in case gravity acts oppositely to what is physically realistic, localized steady waves cannot arise at the surface of a two-dimensional infinitely-deep flow of water so long as the profile lacks self-intersections and cusps. (\eqref{E:injective} prevents the curve from developing self-intersections while $w \in C^{1+\alpha}(\R)$ excludes cusps.) In the Stokes wave setting, incidentally, a non-existence proof based upon duality is found in \cite{Tol02}. 

\begin{proof}
Since $1-2\mu w>0$ everywhere on $\R$ by Corollary~\ref{C:main<->bernoulli}, a bootstrapping argument and a Lewy theorem in \cite{Lewy} lead to that $w$ is real analytic. The proof is nearly identical to that of \cite[Theorem~3.7]{BDT00a} and hence we omit the detail. Accordingly \eqref{E:deep}, or equivalently \eqref{E:bernoulli} by Lemma \ref{P:main<->bernoulli} (a), is written as
\begin{equation}\label{E:bernoulli'} 
|W^*|^2(1-2\mu w)\equiv 1,\end{equation}
where $W=1+\mathcal{R}w' $; recall from \eqref{E:f*} and \eqref{E:R}, respectively, that $W^*$ is the non-tangential limit of $W$ and that $\mathcal{R}w'$ holomorphically extends to $\mathbb{P}$ the boundary function $\H w'+iw'$. Moreover $\text{Re}\,W^*=1+\H w'>0$ everywhere on $\R$ by hypothesis.

\medskip

Let $\mathcal{L}=\mathcal{L}(\mu,w):H^1(\R) \to L^2(\R)$ be the linearization of \eqref{E:deep} at the solution pair $\mu$ and $w$, defined as 
\begin{equation}\label{D:L}
\mathcal{L}v=\H v'-\mu(v+ w\H v'+ v\H w'+\H(vw)').\end{equation}
Clearly $\mathcal{L}w'=0$. Following \cite[Section~4]{Plo} in the finite-depth case or \cite[Section~5]{BDT00a} in the Stokes wave setting we introduce the Plotnikov transformation 
\begin{equation}\label{E:plotnikov}
Pv:=\text{Im}\,(W^*(\mathcal{R}v)^*)=(1+\H w')v+w'\H v.\end{equation}
A straightforward calculation then reveals that $P:H^1(\R) \to H^1(\R)$ is a ho\-meo\-mor\-phism with the inverse
\begin{equation}\label{E:A-1}v=\text{Im}\, \frac{(\mathcal{R}Pv)^*}{W^*}.
\end{equation}
Moreover (see \cite[Theorem~4.2]{Plo} and \cite[Theorem~5.1]{BDT00a}, for instance)
\begin{equation*}\label{E:linear} 
\int^\infty_{-\infty} \mathcal{L}PvPu~dx=\int^\infty_{-\infty} (\H v'-G(x)v)u~dx\end{equation*}
for any $u,v\in H^1(\R)$, where ${\displaystyle G=\text{Im}\,\frac{W^{*\prime}}{W^*}+\mu |W^*|^2(1+\H w')}$. Note that $G(x) \to \mu$ as $|x| \to \infty$ since $w(x),w'(x) \to 0$ and $W^*(x) \to 1$ as $|x| \to \infty$.

\medskip

In view of \eqref{E:A-1} and \eqref{E:bernoulli'} we revamp $\mathcal{L}w'=0$ as
\begin{equation}\label{E:linear'} 
\H v'+(\mu-G(x))v=\mu v,\qquad v=(1-2\mu w)w'\end{equation}
in the class of distributions. In other words, $v$ is an eigenfunction corresponding to the eigenvalue $\mu<0$ of  $\H d/dx+\mu-G(x)$. (The essential spectrum of $\H d/dx+\mu-G(x)$ lies in $(0,\infty)$.) Incidentally $\H d/dx=(-d^2/dx^2)^{1/2}$ is viewed as a relativistic Schr\"odinger operator in one dimension.
Since $G(x)$ is continuous on $\R$ ($G$ is real analytic) and since $\mu -G(x) \to 0$ as $|x| \to \infty$, a result in \cite[Proposition~4.1]{CaMaSi}, for instance, results in that there exists a constant $C>0$ such that $|v(x)| \leq C(1+x^2)^{-1}$ for every $x \in \mathbb{R}$ for some $C>0$. The proof then completes since $1-2\mu w$ is positive everywhere on $\R$ (see the latter in \eqref{E:linear'}).
\end{proof}

%%%%%%%%%%%%%%%%%%%%%%%%%%%%%%%%%%%%%%%%%%%%%%%%%%%%
%%%%%%%%%%%%%%%%%%%%%%%%%%%%%%%%%%%%%%%%%%%%%%%%%%%%

In case gravity acts downwards, on the other hand, $\mu>0$ is contained in the essential spectrum of $\H d/dx+\mu -G(x)$ and hence the preceding proof is not applicable. To attain decay nevertheless, we shall impose a decay condition that guarantees the invertibility of $\H d/dx-G(x)$ but which is milder than the quadratic one, and bootstrap the decay rate. 

\begin{lemma}\label{L:decay+}
If $w\in H^1(\R) \cap C^{1+\alpha}_{1+\epsilon}(\R)$ for some $\alpha \in (0,1)$ and for some $\epsilon \in(0,1]$ satisfies \eqref{E:deep} then $w \in C^{1+\alpha}_2(\R)$, where the weighted H\"older spaces are in \eqref{D:wHolder}. In particular $w'(x)\leq C(1+x^2)^{-1}$ for every $x \in \R$ for some $C>0$.
\end{lemma}

In light of Lemma~\ref{L:non-existence} it proves the first part of Theorem~\ref{T:non-existence}.

\begin{proof}The proof closely resembles that in \cite[Section 4]{Sun97} and hence we merely sketch the detail.

If $w \in C^{1+\alpha}(\R)$ then $\H w'=\partial V/\partial y(\cdot,0) \in C^\alpha (\R)$, where $V$ is a (weak) solution of the boundary value problem
\[ \Delta V=0 \quad \text{in $\P$}, \qquad V=w \quad \text{on $y=0$}\]
and $V(x,y) \to 0$ as $|(x,y)| \to \infty$. Indeed,  
\[ \frac{\partial V}{\partial y}(x,y)=\frac{1}{\pi}\int^\infty_{-\infty} \frac{(x-t)w'(t)}{(x-t)^2+y^2}~dt
\qquad ((x,y)\in \overline{\P})\]by the Poisson integral formula while Privalov's theorem for the Hilbert transform (see Section~\ref{SS:preliminaries}) confirms the regularity. To interpret, $\H d/dx:C^{1+\alpha}(\R) \to C^\alpha(\R)$ is the Dirichlet-to-Neumann operator. Consequently one may associate \eqref{E:deep}, or equivalently \eqref{E:bernoulli} by Lemma \ref{P:main<->bernoulli} (a), with the boundary value problem
\begin{subequations}\label{E:BVP}
\begin{alignat}{2}
\Delta V=&\,0\qquad & &\text{in } \mathbb{P}, \label{E:BVP-laplace}\\
V=\,&w  \qquad & &\text{on }y=0,\label{E:BVP-dirichlet}\\
\frac{1}{V_x^2+(1+V_y)^2}&+2\mu V=1\qquad & &\text{on } y=0, \label{E:BVP-nonlinear}
\end{alignat} 
\end{subequations}
subject to that $V(x,y), \nabla V(x,y) \to 0$ as $|(x,y)| \to \infty$. A straightforward calculation moreover manifests that \eqref{E:BVP-nonlinear} is written as
\[ V_y-\mu V=(\mu V-1/2)(V_x^2+V_y^2)+2\mu VV_y=:G(V, \nabla V) \quad \text{on $y=0$}.\]
To recapitulate, \eqref{E:deep} is reformulated as 
\begin{equation}\label{E:BVP'} 
\H V'(\cdot,0)-\mu V(\cdot,0)=G(V(\cdot,0), \nabla V(\cdot,0)).\end{equation}

\medskip

We claim that if $V(\cdot,0) \in C^{1+\alpha}_{1+\epsilon}(\R)$ satisfies \eqref{E:BVP'} for some $\alpha \in (0,1)$ and for some $\epsilon \in (0,1]$ then $V(\cdot,0) \in C^{1+\alpha}_{1+2\epsilon}(\R)$. That is to say, if a solution of \eqref{E:BVP'} decays like $x^{-1-\epsilon}$ as $|x| \to \infty$ then it gains the additional decay of $x^{-\epsilon}$.

Since $G(V(\cdot,0), \nabla V(\cdot,0))$ behaves quadratically as $|V(\cdot,0)| \to 0$, 
\begin{equation}\label{E:G}
\|G(V(\cdot,0), \nabla V(\cdot,0))\|_{C^{\alpha}_{2(1+\epsilon)}(\R)} \leq 
C\|V(\cdot ,0)\|^2_{C^{1+\alpha}_{1+\epsilon}(\R)}\end{equation}
for some constant $C>0$, suggesting us to study the inhomogeneous  linear equation 
\begin{equation}\label{E:BVP''}
\H V'(\cdot,0)-\mu V(\cdot,0)=G(\cdot), \qquad G \in C^\alpha_{2(1+\epsilon)}(\R). \end{equation}
Upon observing that \eqref{E:BVP''} is the first equation in \cite[(4.19)]{Sun97} we then run the argument in \cite[Section~4, Case~II]{Sun97}, for instance, to show that $V(\cdot,0)\in C^{1+\alpha}_{1+2\epsilon}(\R)$. The idea is that the solution of \eqref{E:BVP''} is given by the formula \[ V(x,0)=
\frac{1}{\mu}\H \left( \int_x^\infty \sin\mu (x-t)G(t)~dt\right)'+\int_x^\infty\sin\mu(x-t)G(t)~dt,\] where 
the kernel associated with $\H d/dx$ decays like $x^{-2}$ as $|x|\to \infty$. This proves the claim. We refer the reader to \cite[Section~4, Case~II]{Sun97} for the detail. 

To summarize, if $w \in C^{1+\alpha}_{1+\epsilon}(\R)$ satisfies \eqref{E:deep} for some $\alpha \in (0,1)$ and for some $\epsilon \in (0,1]$ then $w=V(\cdot,0) \in C^{1+\alpha}_{1+2\epsilon}(\R)$, where $V$ solves \eqref{E:BVP}. We then run a bootstrapping argument in \cite{CrSt}, for instance, to conclude that $w=V(\cdot,0) \in C^{1+\alpha}_2(\R)$. 
\end{proof}

\subsection*{Acknowledgment} The author is supported by the National Science Foundation under grant No. DMS-1008885 and by the University of Illinois at Urbana-Champaign under the Campus Research Board grant No. 11162. She is grateful to anonymous referees for helpful comments.

\bibliographystyle{amsalpha}
\bibliography{steadyWW}

\end{document}